\documentclass[12pt,reqno]{amsart}
\usepackage[utf8]{inputenc} 
\usepackage[LGR,T1]{fontenc}

\title[]{An entropic interpolation proof of the HWI inequality}
\date{July 15, 2018}


\author{Ivan Gentil}
\address{Ivan Gentil. Univ Lyon, Université Claude Bernard Lyon 1, CNRS UMR 5208, Institut Camille Jordan, F-69622 Villeurbanne, France}
\email{gentil@math.univ-lyon1.fr}
 
\author{Christian Léonard}
\address{Christian Léonard. Modal-X. Université Paris Nanterre. Nanterre, France}
\email{leonard@parisnanterre.fr}

\author{Luigia Ripani}
\address{Luigia Ripani. Univ Lyon, Université Claude Bernard Lyon 1, CNRS UMR 5208, Institut Camille Jordan, F-69622 Villeurbanne, France}
\email{ripani@math.univ-lyon1.fr}

\author{Luca Tamanini}
\address{Luca Tamanini. Institut f\"ur Angewandte Mathematik. Universit\"at Bonn. Bonn, Germany}
\email{tamanini@iam.uni-bonn.de}

\usepackage{amssymb, amsmath, amsfonts, mathrsfs, latexsym}
\usepackage{enumerate}
\usepackage{color}
\usepackage{pst-fill,pst-grad,pst-plot,pst-eucl,pstricks-add,pst-node}
\RequirePackage[colorlinks,linkcolor=blue,citecolor=blue,urlcolor=blue]{hyperref}
\usepackage[english]{babel}

\newtheorem{theorem}[equation]{Theorem}
\newtheorem{lemma}[equation]{Lemma}
\newtheorem{proposition}[equation]{Proposition}

\newtheorem{definition}[equation]{Definition}

\newtheorem{assumptions}[equation]{Assumptions}

\theoremstyle{remark}
\newtheorem{remark}[equation]{Remark}
\newtheorem{remarks}[equation]{Remarks}

\numberwithin{equation}{section}

\renewcommand{\H}{\mathbb{H}}

\newcommand{\N}{\mathbb{N}}

\newcommand{\R}{\mathbb{R}}

\newcommand{\sfd}{{\sf d}}

\newcommand{\Id}{{\rm Id}}                          % Identita
\newcommand{\Kliminf}{K\kern-3pt-\kern-2pt\mathop{\rm lim\,inf}\limits}  % Kuratowski liminf di insiemi
\newcommand{\supp}{\mathop{\rm supp}\nolimits}   % supporto 
\renewcommand{\d}{{\mathrm d}}
\newcommand{\dt}{{\d t}}
\newcommand{\ddt}{{\frac \d\dt}}

\newcommand{\restr}[1]{\lower3pt\hbox{$|_{#1}$}} %restrizione funzione

\newcommand{\eps}{\varepsilon}  
\newcommand{\nchi}{\chi}
\newcommand{\prob}[1]{\mathcal P(#1)}                   % misure di probabilita
\newcommand{\probt}[1]{\mathcal P_2(#1)}                   % misure di probabilita con momento secondo finito
\newcommand{\mm}{\mathfrak m}                                %misura di riferimento

\newcommand{\X}{{\rm X}}

\newcommand{\vol}{{\rm Vol}}

\newcommand{\RCD}{{\sf RCD}}

\newcommand{\HS}{{\lower.3ex\hbox{\scriptsize{\sf HS}}}}
\renewcommand{\H}[1]{{\rm Hess}(#1)}

\newcommand{\Ric}{{\rm Ric}}

\newcommand{\hr}{{\sf r}}
\newcommand{\hE}{{\sf E}}
\newcommand{\hP}{{\sf P}}
\newcommand{\hQ}{{\sf Q}}
\newcommand{\hR}{{\sf R}}

\newcommand{\hL}{\mathcal{L}}
\newcommand{\hT}{\mathcal{T}}

\newcommand{\vf}{\overrightarrow{v}}
\newcommand{\vb}{\overleftarrow{v}}
\newcommand{\vc}[1]{v ^{ \mathrm{cu},#1}}
\newcommand{\vo}[1]{v ^{ \mathrm{os},#1}}

\begin{document}

\begin{abstract} 
We present a pathwise proof of the HWI inequality which is based on entropic interpolations rather than displacement ones. Unlike the latter, entropic interpolations are regular both in space and time. Consequently, our approach is closer to the Otto-Villani heuristics, presented in the first part of the article \cite{OttoVillani00}, than the original rigorous proof presented in the second part of \cite{OttoVillani00}.
\end{abstract}

\maketitle

\tableofcontents

% ************* corps du texte ****************************

\section*{Introduction}

In a seminal paper \cite{OttoVillani00}, Otto and Villani obtained a powerful functional inequality relating the relative entropy $H(\cdot\,|\,\mm)$ with respect to some reference measure $\mm \in \probt{\X}$, the quadratic transport cost $W_2^2(\cdot,\mm)$ and the Fisher information $I(\cdot\,|\,\mm)$. This so-called HWI inequality roughly states that: $H\le W \sqrt{I}- \kappa W^2/2,$ where the real parameter $\kappa$ is a curvature lower bound associated to $\mm$, see \eqref{eq-cd(k,inf)}, \eqref{eq-cd(k,n)} and Theorem \ref{res-h05} below for the exact statement and its well-known consequences in terms of Talagrand and logarithmic Sobolev inequalities. 

The first part of Otto and Villani's article is dedicated to a heuristic proof based on Otto calculus, see \cite{Otto01,Villani09}, where one formally equips the set of probability measures with the Riemannian-like distance $W_2$ and where McCann displacement interpolations are interpreted as geodesics. Since these interpolations suffer from a lack of regularity, the first and second order time derivatives along them are only formal. Consequently, although heuristics led to the right conjecture, the authors presented an alternative rigorous proof based on a significantly different approach.

In the present article, a new proof of the HWI inequality is proposed. The main idea is to replace the `irregular' McCann interpolation $(\mu_t)$ between two probability measures $\mu_0$ and $\mu_1$ by a family of `smooth' curves of measures $(\mu_t^\eps)$, called `entropic interpolations' (see Definition \ref{def-h01} below), where $\eps > 0$ is a small fluctuation parameter such that $\mu_t^\eps$ converges to $\mu_t$ narrowly as $\eps \downarrow 0$. Otto and Villani's heuristics apply rigorously to $(\mu_t^\eps)$, so that it remains to let $\eps$ tend down to zero to obtain the desired result.

The paper is structured as follows. Section \ref{sec-HWI-statement} is dedicated to the statement of the HWI inequality. Basic material about entropic interpolations which is needed for the proof is gathered at Section \ref{sec-entint}. Finally the proof of the inequality is done at Section \ref{sec-HWI-proof}; its core is Lemma \ref{res-h09} which is the analogue of Otto and Villani's heuristic approach. In Section \ref{sec-remarks} some comments about possible extensions and simplifications of our approach are collected.

\section{Statement of  the HWI inequality}
\label{sec-HWI-statement}

Before stating the HWI inequality at Proposition \ref{res-h04} and Theorem \ref{res-h05} below, we need to make clear the framework we shall work within and introduce the quantities $H$, $W$ and $I$.

\subsection*{Setting \ref{setting}} \label{setting}

Let $(\X,\sfd,\mm)$ be:
\begin{enumerate}[(a)]
\item either $(\R^n,|\cdot|,\mm)$, where $|\cdot|$ is the Euclidean distance and the reference measure $\mm$ is defined as
\begin{align}\label{eq-01a}
\mm := e^{-V}\mathcal{L}^n
\end{align}
with $\mathcal{L}^n$ the $n$-dimensional Lebesgue measure and $V : \R^n \to [0,\infty)$ satisfying the following hypotheses: it belongs to $C^\infty(\R^n)$, is such that $\mm$ is a probability measure and
\begin{equation}\label{eq-cd(k,inf)}
\H V \geq \kappa\Id
\end{equation}
for some $\kappa \in \R$;
\item or $(M,\sfd_g,\mm)$, where $M$ is a smooth Riemannian manifold without boundary and with metric tensor $g$, $\sfd_g$ is the induced distance and $\mm$ is given by
\begin{align}\label{eq-01b}
\mm := e^{-V}\vol
\end{align}
with $\vol$ the volume measure on $M$ and $V : M \to [0,\infty)$ satisfying the following hypotheses: it belongs to $C^\infty(\R^n)$, is such that $\mm$ is a probability measure and, for some $\kappa \in \R$ and $N \geq n = \dim(M)$, the Bakry-\'Emery Ricci tensor $\Ric_{V,N}$ satisfies the lower bound
\begin{equation}\label{eq-cd(k,n)}
\Ric_{V,N} := \Ric_g - (N-n)\frac{\H{e^{-\frac{1}{N-n}V}}}{e^{-\frac{1}{N-n}V}} \geq \kappa g.
\end{equation}
\end{enumerate}

\subsection*{Relative entropy}

For any two probability measures $p$ and $r$ on a measurable space $Z$ the relative entropy of $p$ with respect to $r$ is defined by
\[
H(p\,|\,r) := \int_Z \log\Big(\frac{\d p}{\d r}\Big) \d p \in [0,\infty],
\]
where it is understood that this quantity is infinite when $p$ is not absolutely continuous with respect to $r$. In our case, $Z$ will be $\X$, $\X \times \X$ or $C([0,1],\X)$.

\subsection*{Quadratic transport cost}

By $\prob\X$ we shall denote the space of Borel probability measures on $\X$ and by $\probt\X$ the subclass of those with finite second moment, namely all $\mu \in \prob\X$ such that $\int_\X \sfd^2(\cdot,x)\d\mu < \infty$ for some (and thus all) $x \in \X$. With this said, the squared Wasserstein distance between $\mu,\nu \in \probt\X$ is defined as
\[
W_2^2(\mu,\nu) := \inf_{\pi} \int_{\X \times \X} \sfd^2(x,y)\, \pi(\d x\d y)
\]
where the infimum runs through all the couplings $\pi \in \prob{\X \times \X}$ of $\mu$ and $\nu$, that is $\pi(\d x \times \X) = \mu(\d x)$ and $\pi(\X \times \d y) = \nu(\d y)$.

\subsection*{Fisher information}
The Fisher information of $\mu \in \prob\X$ with respect to $\mm$ is defined by
\[
I(\mu\,|\,\mm) := 4\int_\X |\nabla\sqrt{\rho}|^2\,\d\mm = \int_{\{\rho > 0\}}\frac{|\nabla\rho|^2}{\rho}\,\d\mm \qquad \text{ if } \mu = \rho\mm,\,\sqrt{\rho} \in W^{1,2}(\X)
\]
and $+\infty$ otherwise. Up to identify $\mu$ with its density, the Fisher information is lower semicontinuous with respect to the weak topology of $L^1(\mm)$ (see for instance \cite{AmbrosioGigliSavare11}).

\medskip

With this premise, the statement of the HWI$^*$ inequality is

\begin{proposition}[HWI$^*$ inequality] \label{res-h04}
Let $(\X,\sfd,\mm)$ be as in  Setting  \ref{setting}. Then for any $\mu_0,\mu_1 \in \probt\X$ such that $H(\mu_0\,|\,\mm) < \infty$,
\[
H(\mu_1\,|\,\mm) - H(\mu_0\,|\,\mm) \leq W_2(\mu_0,\mu_1)\sqrt{I(\mu_1\,|\,\mm)} - \frac{\kappa}{2} W_2^2(\mu_0,\mu_1).
\]
\end{proposition}

In this case it is said that the reference measure $\mm$ satisfies the HWI inequality. 

As already shown by Otto and Villani in \cite{OttoVillani00}, different choices of $\mu_0$ and $\mu_1$ in Proposition \ref{res-h04} entail three important consequences collected here below.

\begin{theorem}\label{res-h05}
Let $(\X,\sfd,\mm)$ be as in  Setting  \ref{setting} with the further assumption that $\mm \in \probt\X$. Then the following inequalities are satisfied:
\begin{enumerate}[(a)]
\item
HWI inequality: 
\begin{equation*}
H(\nu\,|\,\mm) \leq W_2(\nu,\mm) \sqrt{I(\nu\,|\,\mm)}- \frac{\kappa}{2}W_2^2(\nu,\mm), \qquad \forall \nu \in \probt\X;
\end{equation*}
\item
Talagrand inequality:
\begin{equation*}
\frac{\kappa}{2}W_2^2(\nu,\mm) \leq H(\nu\,|\,\mm), \qquad \forall \nu \in \probt\X;
\end{equation*}
\item
Logarithmic Sobolev inequality: assume that $\kappa>0$, then
\begin{equation*}
H(\nu\,|\,\mm) \leq \frac{1}{2\kappa}I(\nu\,|\,\mm), \qquad \forall \nu \in \prob\X.
\end{equation*}
\end{enumerate}
\end{theorem}

\begin{proof}
First of all, since $\mm \in \probt\X$, it follows that $W_2(\nu,\mm)$ is finite for all $\nu \in \probt\X$.
\begin{enumerate}[(a)]
\item
The HWI inequality is obtained by choosing $\mu_0 = \mm,\mu_1 = \nu$.
\item
The Talagrand inequality is obtained by choosing $\mu_0 = \nu,\mu_1 = \mm$.
\item
When $\kappa>0$, the logarithmic Sobolev inequality with $\nu \in \probt\X$ follows by taking the supremum with respect to $W_2$ in the right-hand side of the HWI inequality (a). To extend this result to the case where $\nu \in \prob\X$, a standard approximation argument (carried out for instance in Lemma \ref{res-h01} below) is sufficient.
\end{enumerate}
\end{proof}

\begin{remarks}\ \begin{enumerate}[(a)] \label{rem-01}
\item
When $\kappa>0$, any $\nu \in \prob\X$ such that $H(\nu\,|\,\mm) < \infty$ stands in $\probt\X$. This follows from the variational representation of the relative entropy as
\[
H(\nu\,|\,\mm) = \sup_f\bigg\{\int_\X f\,\d\nu - \log\int_\X e^f\,\d\mm \,:\, \int_\X e^f\,\d\mm < \infty,\, \int_\X f^-\,\d\nu < \infty \bigg\},
\]
where $f^- := \max\{-f,0\}$ (see for instance \cite{Leonard14b} for a proof). If we choose $f = \alpha\sfd^2(\cdot,x)$ for some $x \in \X$ and $ 0 < \alpha < \kappa$, then $\int_\X e^f\,\d\mm < \infty$ holds: in Setting  \ref{setting}-(a) this is due to \eqref{eq-cd(k,inf)} whereas in Setting  \ref{setting}-(b) to \eqref{eq-cd(k,n)} and the Bishop-Gromov inequality. Therefore
\[
\int_\X \sfd^2(\cdot,x)\,\d\nu \leq \alpha^{-1}\Big( H(\nu\,|\,\mm) + \log\int_M e^{\alpha \sfd^2(\cdot,x)}\,\d\mm\Big),
\]
whence the claim. In particular, $\mm \in \probt\X$.
\item
Talagrand inequality (b) is irrelevant when $\kappa \leq 0$. When $\kappa>0,$ in view of previous remark it extends to all $\nu \in \prob\X$ provided that one sets $W_2(\nu,\mm)= \infty$ when $\nu$ does not belong to $\probt\X$. 
\item
It follows from the logarithmic Sobolev inequality that when \eqref{eq-cd(k,inf)} or \eqref{eq-cd(k,n)} holds with some $\kappa>0$, any $\nu \in \prob\X$ such that $H(\nu\,|\,\mm) = \infty$ satisfies $I(\nu\,|\,\mm) = \infty$. Similarly, it follows from the HWI inequality that when \eqref{eq-cd(k,inf)} or \eqref{eq-cd(k,n)} is only supposed to hold with $\kappa$ real, as soon as $\nu \in \probt\X$, then $H(\nu\,|\,\mm) = \infty$ implies that $I(\nu\,|\,\mm) = \infty$.
\end{enumerate}
\end{remarks}

\section{Entropic interpolations}
\label{sec-entint}

In this section we propose a short and self-contained presentation of entropic interpolations. The purpose is twofold: to provide the reader with those notions and results that will be frequently used later on and discuss their physical interpretation via Nelson's dynamical view of diffusion processes. For sake of simplicity, the latter will be carried out in the more familiar Euclidean setting.

\subsection*{Notations}

Let $X = (X_t)_{0 \leq t \leq 1}$, $X_t : \Omega \to \X$ with $\Omega := C([0,1],\X)$ be the canonical process, defined by
\[
X_t(\omega) := \omega_t, \qquad \forall\, \omega \in \Omega,\forall 0\, \leq t \leq 1.
\]
For any path measure $\hQ \in \mathcal{P}(\Omega)$ and each $0 \leq t \leq 1$, we denote by $\hQ_t := (X_t)_{\#}\hQ \in \prob\X$ the $t$-th marginal of $\hQ$, that is the law of the position $X_t$ at time $t$ of the random path $X$ under $\hQ$. Moreover, for any $0 \leq s,t \leq 1$, we shall denote by $\hQ_{st}$ the joint law of $X_s$ and $X_t$ under $\hQ$, namely $\hQ_{st} := (X_s,X_t)_{\#}\hQ$.

As reference path measure $\hR \in \prob\Omega$ we consider the law of the Markov diffusion process with generator
\[
\hL := \frac{1}{2}(\Delta -\nabla V \cdot\nabla)
\]
with initial law $\hR_0 = \mm$, where the potential $V$ appears at \eqref{eq-01a}, \eqref{eq-01b}, $\Delta$ is the Laplace-Beltrami operator on $\X$ and $\nabla$ the Levi-Civita connection associated to the metric $g$ (in Setting  \ref{setting}-(a) they are nothing but the standard Laplacian and gradient). It is well-known that $\hR$ is a reversible Markov measure with reversing measure $\mm$. In particular it is stationary, that is $\hR_t = \mm$ for all $0 \leq t \leq 1$. For any $\eps>0$ we denote by $X^\eps$ the time-rescaled process defined by $X^\eps_t := X_{\eps t}$, $0 \leq t \leq 1$, and by $\hR^\eps := (X^\eps)_{\#}\hR$ the corresponding path measure. The parameter $\eps$ is meant to tend to zero so that $\hR^\eps$ is a slowed down version of $\hR$, whose generator is
\[
\hL^\eps = \eps \hL = \frac{\eps}{2}(\Delta -\nabla V \cdot\nabla).
\]
For any $\eps>0$, as a time rescaling of $\hR$, $\hR^\eps$ is also $\mm$-reversible, so that in particular $\hR^\eps_t = \mm$ for all $t$.

The 1-parameter semigroup associated to $\hL$ will be denoted by $(\hT_t)$ and, in a completely analogous way, $(\hT^\eps_t)$ the one associated to $\hL^\eps$; notice that $\hT^\eps_t = \hT_{\eps t}$ for all $t \geq 0$. Within  Setting  \ref{setting} it is well-known (see for instance \cite{Grigoryan09}) that there exists a unique heat kernel $\hr_t(x,y)$ associated to $\hL$ which is a smooth function on $(0,\infty) \times \X \times \X$. Therefore, the semigroup $(\hT_t)$ can be represented by
\begin{equation}\label{eq:semigroup}
\hT_t f(x) = \hE_\hR[f(X_0) \,|\, X_t = x] = \int_\X f(x)\hr_t(x,y)\,\mm(\d y)
\end{equation}
for all $f \in L^\infty(\mm)$. Let us also recall that, in conjunction with \eqref{eq-cd(k,inf)} or \eqref{eq-cd(k,n)}, $(\hT_t)$ enjoys the Bakry-\'Emery contraction estimate
\begin{equation}\label{eq:bakry}
|\nabla\hT_t f|^2 \leq e^{-2\kappa t}\hT_t(|\nabla f|^2) \qquad \forall f \in C_c^\infty(\X),\, t \geq 0.
\end{equation}
For its proof as well as for all the regularizing properties of $\hT$ that will be used throughout the paper, we address the reader to \cite{BakryGentilLedoux14}.

Finally, recall that the operators
\[
\Gamma(f,g) := \hL(fg) - f\hL g - g\hL f \qquad \Gamma_2(f,g) := \hL\Gamma(f,g) - \Gamma(f,\hL g) - \Gamma(g,\hL f),
\]
defined for all $f,g \in C^\infty_c(\X)$, are naturally associated  to $\hL$. As it is not difficult to see, the drift $V$ in $\mm$ does not affect $\Gamma$, since $\Gamma(f,g) = \langle\nabla f,\nabla g\rangle$. It is worth mentioning that, with respect to the standard definition provided in \cite{BakryGentilLedoux14}, here $\Gamma$ and $\Gamma_2$ are not divided by 2, as the factor 1/2 already appears in $\hL$, which thus corresponds to an SDE driven by a standard Brownian motion.

\subsection*{The Schrödinger problem}

Let $\mu_0,\mu_1 \in \prob\X$ be two probability measures: the Schr\"odinger problem associated with $\hR^\eps,\mu_0,\mu_1$ is defined by
\begin{equation}\label{eq-h06}
\eps H(\hP\,|\,\hR^\eps)\to\min; \qquad \hP \in \mathcal{P}(\Omega)\,:\, \hP_0 = \mu_0, \hP_1= \mu_1
\tag{S$_{\eps}$}
\end{equation}
and its value is called `entropic cost'. As a strictly convex minimization problem, it admits at most one solution. 

\begin{definition}[$\hR^\eps$-entropic interpolation] \label{def-h01}
The solution $\hP^\eps$ of \eqref{eq-h06}, if it exists, is called the $\hR^\eps$-entropic bridge between $\mu_0$ and $\mu_1$. The $\hR^\eps$-entropic interpolation $(\mu^\eps_t)$ between $\mu_0$ and $\mu_1$ is defined as the time marginal flow of the solution $\hP^\eps$, namely
\[
\mu^\eps_t := \hP^\eps_t, \qquad 0 \leq t \leq 1.
\]
\end{definition}

The name `entropic interpolation' stems from the connection with displacement interpolation. Indeed, it is known from \cite{Mikami04} that
\begin{equation}\label{eq:converge}
\lim_{\eps \downarrow 0} \,\,\inf \eqref{eq-h06} = \frac{1}{2} W_2^2(\mu_0,\mu_1).
\end{equation}
This limit is a  consequence of a more general result asserting that
 \eqref{eq-h06} converges to the quadratic Monge-Kantorovich problem as $\eps \downarrow 0$ in the sense of $\Gamma$-convergence, see  \cite{Leonard12}.
\\
As shown in \cite{Leonard14}, if \eqref{eq-h06} admits the solution $\hP^\eps$, then there exist two non-negative measurable functions $f^\eps,g^\eps : \X \to [0,\infty)$ such that $\hP^\eps = f^\eps(X_0)g^\eps(X_1)\,\hR^\eps$ and defining  
\begin{align}\label{eq-h23b}
f^\eps_t := \mathcal{T}^\eps_t f^\eps, \quad g^\eps_t := \mathcal{T}^\eps_{1-t}g^\eps, \qquad 0 \leq t \leq 1
\end{align}
we obtain $\hP^\eps_t(\d x) = \mu_t^\eps(\d x) = f^\eps_t(x)g^\eps_t(x)\,\mm(\d x)$. By computing the endpoint marginals of $\hP^\eps$ we obtain the following two conditions
\begin{equation}\label{eq-h22}
\rho_0 := \frac{\d\mu_0}{\d\mm} = f^\eps\, g^\eps_0 \qquad\qquad \rho_1 := \frac{\d\mu_1}{\d\mm} = f^\eps_1\, g^\eps
\end{equation}
usually known as `Schr\"odinger system': indeed, if we interprete them as a nonlinear system where the unknowns are $f^\eps$ and $g^\eps$, then the (unique up to an obvious multiplicative rescaling) solution completely determines $\hP^\eps$ (see for instance \cite{Leonard14}). As far as the convergence of entropic interpolations towards displacement ones is investigated, the following functions
\[
\varphi^\eps_t = \eps\log f^\eps_t \qquad\qquad \psi^\eps_t := \eps\log g^\eps_t
\]
are of special interest. We also set $\varphi^\eps := \eps\log f^\eps$ in $\supp(\mu_0)$ and $\psi^\eps := \eps\log g^\eps$ in $\supp(\mu_1)$. They are called Schrödinger potentials, in connection with Kantorovich ones.

\subsection*{Existence and regularity results}

Let us now derive a criterion, in terms of the endpoint marginals $\mu_0$ and $\mu_1$, for the existence of regular functions $f^\eps,g^\eps$ with well-defined Fisher information solving the Schrödinger system \eqref{eq-h22}. As noticed in \cite{GigTam17}, \cite{GigTam18} the regularity (smoothness and integrability) of $\mu_0$ (resp.\ $\mu_1$) is inherited by $f^\eps$ (resp.\ $g^\eps$). In the next result we extend this property: if $I(\mu_0\,|\,\mm)$ is finite, then so is $I(f^\eps\mm\,|\,\mm)$ and analogously with $\mu_1,g^\eps$.

\begin{proposition}\label{res-h02}
Let $(\X,\sfd,\mm)$ be as in  Setting  \ref{setting}, $\eps > 0$ and consider two probability measures $\mu_0 = \rho_0\mm,\mu_1 = \rho_1\mm \in \prob\X$ with compact supports. Then the following hold: 
\begin{enumerate}[(a)]
\item
The Schrödinger problem \eqref{eq-h06} admits a solution with finite entropy if and only if $H(\mu_0\,|\,\mm)$, $H(\mu_1\,|\,\mm)< \infty$. This solution is unique and the $\hR^\eps$-entropic interpolation between $\mu_0$ and $\mu_1$ exists.

\item
Suppose in addition that $\rho_0,\rho_1 \in L^\infty(\mm)$.
\begin{enumerate}[(i)]
\item
Then $f^\eps,g^\eps \in L^\infty(\mm)$.

\item
For any $0<t<1$ the functions $f^\eps_t,g^\eps_t,\rho^\eps_t$ as well as $f^\eps_1$ and $g^\eps_0$ belong to $C^\infty(\X) \cap L^\infty(\mm)$.

\item
For any $k \in \N \cup \{\infty\}$, if $\rho_0 \in C^k(\X)$ (resp.\ $\rho_1$), then the function $f^\eps$ (resp.\ $g^\eps$) also belongs to $C^k(\X)$.
\end{enumerate}

\item
In addition to item (b), suppose that $I(\mu_0\,|\,\mm)$ is finite (resp.\ $I(\mu_1\,|\,\mm)$). Then so is $I(f^\eps\mm\,|\,\mm)$ (resp.\ $I(g^\eps\mm\,|\,\mm)$).
\end{enumerate}
\end{proposition}

\begin{proof}
For (a) and (b-i) see \cite[Proposition 2.1]{GigTam18}. As regards (b-ii), the fact that $f^\eps \in L^\infty(\mm)$, $\hr_t \in C^\infty(\X \times \X)$ and \eqref{eq:semigroup} imply $f_t^\eps \in C^\infty(\X)$ for all $0 < t \leq 1$, while the maximum principle ensures that $f_t^\eps \in L^\infty(\mm)$; the statements for $g^\eps_t$ and  $\rho_t^\eps = f_t^\eps g_t^\eps$ follow by the same reason.

\noindent{\bf (b-iii)}. Notice that the first equation in the Schrödinger system \eqref{eq-h22} can be rewritten as
\[
f^\eps = \frac{\rho_0}{\hT^\eps_1 g^\eps}.
\]
Since $\hT^\eps$ is positivity improving, $\hT^\eps_1 g^\eps$ is smooth and $\rho_0 \in C^k(\X)$ with compact support, the conclusion follows. A similar argument holds for $g^\eps$.

\noindent{\bf (c)}. Observe that the Schrödinger system \eqref{eq-h22} and $g_0^\eps > 0$ (as $\hT^\eps$ is positivity improving) force $f^\eps$ to have the same support as $\rho_0$. Thus $f^\eps$ has compact support and, as a consequence, $g_0^\eps \geq c > 0$ in $\supp(f^\eps)$ for some $c$. This remark and the chain rule allow us to say that
\[
\begin{split}
c\frac{|\nabla f^\eps|^2}{f^\eps} & \leq g_0^\eps \frac{|\nabla f^\eps|^2}{f^\eps} \leq g_0^\eps \frac{|\nabla f^\eps|^2}{f^\eps} + f^\eps \frac{|\nabla g_0^\eps|^2}{g_0^\eps} = \frac{|\nabla\rho_0|^2}{\rho_0} - 2\langle\nabla f^\eps,\nabla g_0^\eps\rangle \\ & \leq \frac{|\nabla\rho_0|^2}{\rho_0} + |\nabla f^\eps|^2 + |\nabla g_0^\eps|^2,
\end{split}
\]
so that it remains to prove the integrability of the right-hand side. The first term is integrable by assumption, the third one by the regularization properties of $\hT^\eps$, while for the second one notice that $I(\mu_0\,|\,\mm) < \infty$ and $\rho_0 \in L^\infty(\mm)$ imply $|\nabla\rho_0| \in L^2(\mm)$; plugging this information into
\[
c|\nabla f^\eps| \leq g_0^\eps |\nabla f^\eps| \leq |\nabla\rho_0| + f^\eps|\nabla g_0^\eps| 
\]
we get the desired conclusion.
\end{proof}

For this reason we formulate the following

\begin{assumptions}[Hypotheses on $\mu_0,\mu_1$] \label{ass-01}
The endpoint marginals $\mu_0,\mu_1 \in \probt\X$ are such that $\mu_0,\mu_1$ have compact supports, $H(\mu_0\,|\,\mm), H(\mu_1\,|\,\mm), I(\mu_1\,|\,\mm) < \infty$ and their densities $\rho_0,\rho_1$ belong to $C^\infty(\X)$.
\end{assumptions}

\subsection*{A dynamical viewpoint}

As concerns the evolution of entropic interpolations and Schrödinger potentials, let us first notice that under Assumptions \ref{ass-01}, by the very definition \eqref{eq-h23b},  item (b-iii) of Proposition \ref{res-h02} and the fact that $\hr_t(x,y) \in C^\infty((0,\infty) \times \X^2)$ we deduce that $f_t^\eps$ and $g_t^\eps$ are smooth on $[0,1] \times \X$ and solve
\begin{equation}\label{eq:heat}
\partial_t f_t^\eps = \eps\hL f_t^\eps \qquad\qquad -\partial_t g_t^\eps = \eps\hL g_t^\eps
\end{equation}
in the classical sense. Moreover, if we look at $(f_t^\eps),(g_t^\eps)$ as curves parametrized by $t$ with values in $W^{1,2}(\X,\mm)$, they belong to the set $AC([0,1],W^{1,2}(\X,\mm))$ of all absolutely continuous functions from $[0,1]$ to $W^{1,2}(\X,\mm)$. As a consequence the PDEs above hold also when $\partial_t f_t^\eps,\partial_t g_t^\eps$ are seen as strong $W^{1,2}$-limits. 

Relying on that, it follows that the Schrödinger potentials $\varphi_t^\eps,\psi_t^\eps$ are smooth on $(0,1] \times \X$ and $[0,1) \times \X$ and solve forward and backward Hamilton-Jacobi-Bellman equations respectively, i.e.
\begin{equation}\label{eq:hjb}
\partial_t\varphi_t^\eps = \frac{1}{2}|\nabla\varphi_t^\eps|^2 + \eps\hL \varphi_t^\eps \qquad\qquad \partial_t\psi_t^\eps = \frac{1}{2}|\nabla\psi_t^\eps|^2 + \eps\hL \psi_t^\eps,
\end{equation}
while for $(\rho_t^\eps)$ the continuity equation
\begin{equation}\label{eq:continuity}
\partial\rho_t^\eps + {\rm div}_\mm(\nabla\vartheta_t^\eps\,\rho_t^\eps) = 0
\end{equation}
is satisfied in $(0,1) \times \X$, where $\vartheta_t^\eps := (\psi_t^\eps - \varphi_t^\eps)/2$ and ${\rm div}_\mm$ denotes the divergence with respect to $\mm$, i.e.\ the opposite of the adjoint of the differential in $W^{1,2}(\X,\mm)$. This last PDE is strongly linked to the dynamical representation of the entropic cost $\inf\eqref{eq-h06}$, namely
\begin{equation}\label{eq-h05}
\eps \inf\eqref{eq-h06} = \frac{\eps}{2}\Big(H(\mu_0\,|\,\mm) + H(\mu_1\,|\,\mm)\Big) + \int_{[0,1] \times \X}\frac{|\nabla\vartheta_t^\eps|^2}{2}\dt\d\mu_t^\eps + \frac{\eps^2}{8} \int_0^1 I(\mu_t^\eps\,|\,\mm)\,\dt,
\end{equation}
shown in \cite{CGP14,GLR15} for Setting  \ref{setting}-(a) and in \cite{GigTam18b} for rather general metric measure spaces including Setting  \ref{setting}-(b). By the very definition of $\vartheta_t^\eps$ and since $\eps\log\rho_t^\eps = \varphi_t^\eps + \psi_t^\eps$, this implies
\begin{equation}\label{eq:integrability}
\int_{[0,1] \times \X}\frac{|\nabla\varphi_t^\eps|^2}{2}\dt\d\mu_t^\eps, \int_{[0,1] \times \X}\frac{|\nabla\psi_t^\eps|^2}{2}\dt\d\mu_t^\eps < \infty.
\end{equation}

\begin{remark}\label{rem-02}
The regularity of Schrödinger potentials comes from the one of $f_t^\eps,g_t^\eps$, the fact that $f_t^\eps,g_t^\eps$ are everywhere positive and the logarithm is smooth on $(0,\infty)$. However, Schrödinger potentials are not integrable in general, as $f_t^\eps,g_t^\eps$ can be arbitrarily close to 0. Thus we cannot study their behaviour as curves with values into some $L^p(\mm)$ space. This also explains why the regularity of $\varphi_t^\eps$ (resp.\ $\psi_t^\eps$) does not extend up to $t=0$ (resp.\ $t=1$).
\end{remark}

\subsection*{A physical interpretation}

In the Euclidean framework of Setting  \ref{setting}-(a) it is possible to make a bridge between what is presented so far and Nelson's formalism \cite{Nelson67}, thus providing a physical motivation for some results stated above and a further perspective on some objects.

Following \cite{Nelson67} we define the forward and backward velocities of the Markov measure $\hP$ for any $x \in \R^n$ and $0 \leq t <1$, $0 < t \leq 1$ respectively by
\[
\vf^\hP_t(x) := \lim _{h \to 0^+} \hE_\hP \left( \frac{X_{t+h} - X_t}{h}\mid X_t=x \right) \qquad \vb^\hP_t(x) := \lim _{h\to 0^+} \hE_\hP \left( \frac{X_{t-h} - X_t}{h}\mid X_t = x\right)
\]
when these limits are meaningful. The current velocity is defined, for any $0<t<1$ and $x \in \R^n$, by
\[
\vc{\hP}_t(x) := \lim _{ h\to 0^+} \hE_\hP \left( \frac{X _{ t+h}-X _{ t-h}}{2h}\mid X_t=x\right)
\]
and the osmotic velocity by
\[
\vo{\hP}_t(x) := \lim _{ h\to 0^+} \hE_\hP \left( \frac{X _{ t+h}-2X_t+X _{ t-h}}{2h}\mid X_t=x\right).
\]
We immediately see that
\[
\left\{
\begin{array}{l}
\vf^P = \vc P+\vo P,\\
\vb^P = -\vc P+\vo P,
\end{array}\right.
\qquad \textrm{and} \qquad
\left\{
\begin{array}{l}
\vc P = (\vf^P-\vb^P)/2,\\
\vo P = (\vf^P+\vb^P)/2.
\end{array}\right.
\]
For $\hR^\eps$, it is easily seen that
\[
\vc{\hR^\eps}=0,
\qquad \vf ^{\hR^\eps}=\vb ^{\hR^\eps}=\vo{\hR^\eps}= -\frac{\eps}{2}\nabla V
\]
whereas for $\hP^\eps$
\[
\left\{\begin{array}{l}
\vf^{\hP^\eps}_t = \nabla\psi_t^\eps + \frac{\eps}{2}\nabla V, \\
\vb^{\hP^\eps}_t = \nabla\varphi_t^\eps - \frac{\eps}{2}\nabla V,
\end{array}\right.
\qquad \textrm{and} \qquad
\left\{\begin{array}{l}
\vc{\hP^\eps}_t = \nabla\vartheta_t^\eps - \frac{\eps}{2}\nabla V, \\
\vo{\hP^\eps}_t = \frac{\eps}{2}\nabla\log\rho_t^\eps.
\end{array}\right.
\]
This allows to rewrite the continuity equation \eqref{eq:continuity} as
\[
\partial_t\rho_t^\eps + {\rm div}(\vc{\hP^\eps} \rho_t^\eps) = 0
\]
where now ${\rm div}$ denotes the divergence with respect to $\mathcal{L}^n$ if in Setting  \ref{setting}-(a) or $\vol$ if in Setting  \ref{setting}-(b). This is perfectly coherent with \eqref{eq:continuity} since ${\rm div}_\mm(w) = {\rm div}(w) - \frac{\eps}{2}\nabla V \cdot w$ for any vector field $w$. Furthermore, \eqref{eq-h05} becomes
\[
\eps \inf\eqref{eq-h06} = \frac{\eps}{2}\Big(H(\mu_0\,|\,\mm) + H(\mu_1\,|\,\mm)\Big) + \frac{1}{2}\int_{[0,1] \times \X}\Big(|\vc{\hP^\eps}|^2 + |\vo{\hP^\eps}_t|^2\Big)\dt\d\mu_t^\eps.
\]

\section{Proof of Proposition \ref{res-h04}}
\label{sec-HWI-proof}

We need to state some preliminary lemmas before completing the proof of Proposition \ref{res-h04} at page \pageref{page-proof}. Throughout the whole section we shall assume to work within $(\X,\sfd,\mm)$ as in  Setting  \ref{setting}.

\subsection*{Auxiliary lemmas}

Let us start with an approximation result.

\begin{lemma}\label{res-h01}
Let $\mu \in \probt\X$ with $H(\mu\,|\,\mm) < \infty$. Then:
\begin{enumerate}[(a)]
\item there exists a sequence $(\mu_n) \subset \probt\X$ with $\mu_n = \rho_n\mm$ and $\rho_n \in C^\infty_c(\X)$ such that $W_2(\mu_n,\nu) \to W_2(\mu,\nu)$ for all $\nu \in \probt\X$ and $H(\mu_n\,|\,\mm) \to H(\mu\,|\,\mm)$ as $n \to \infty$;
\item if in addition $I(\mu\,|\,\mm) < \infty$, then there exists a sequence $(\mu'_n) \subset \probt\X$ with $\mu_n = \rho_n\mm$, $\rho_n \in C^\infty_c(\X)$ such that $W_2(\mu'_n,\nu) \to W_2(\mu,\nu)$ for all $\nu \in \probt\X$, $H(\mu'_n\,|\,\mm) \to H(\mu\,|\,\mm)$ and $I(\mu'_n\,|\,\mm) \to I(\mu\,|\,\mm)$ as $n \to \infty$.
\end{enumerate}
\end{lemma}

\begin{proof}
Let us write $\mu = \rho\,\mm$ and, as a first step, let us prove that both in (a) and (b) it is possible to find a sequence of measures with smooth densities converging to $\mu$ in the desired sense. This can be proved by defining for $\eps > 0$
\[
\mu_\eps := \rho_\eps\,\mm \qquad \textrm{with } \rho_\eps := \hT_\eps\rho,
\]
which clearly have smooth densities by the regularizing properties of $(\hT_\eps)$. The convergence of $W_2(\mu_\eps,\nu)$, $H(\mu_\eps\,|\,\mm)$ and $I(\mu_\eps\,|\,\mm)$ to $W_2(\mu,\nu)$, $H(\mu\,|\,\mm)$ and $I(\mu\,|\,\mm)$ respectively as $\eps \downarrow 0$ is now a well-known fact in the theory of gradient flows (see for instance Theorem 2.4.15 and Remark 2.4.16 in \cite{AmbrosioGigliSavare08} in conjunction with the fact that the squared slope of the entropy is the Fisher information, as proved in \cite{AmbrosioGigliSavare11}).

Thus, it is not restrictive to suppose that $\mu$ has smooth density. Under this new assumption, let us prove that we can find a sequence of measures with compact supports and smooth densities converging to $\mu$ in the desired sense. To this aim define
\[
\mu_n := \alpha_n\rho_n\, \mm \qquad \textrm{with } \rho_n := \nchi_n^2\rho,
\]
where $\alpha_n$ is the renormalization constant and $\nchi_n$ is a smooth cut-off function with support in $B_{n+1}(x)$, for some $x \in \X$, $\nchi_n(x) = 1$ and Lipschitz constant controlled by $C/n$, where $C > 1$ does not depend on $n$ (see e.g.\,\cite{AFLMR07} for a proof of the existence of such cut-off functions). By dominated convergence it is not difficult to see that $W_2(\mu,\mu_n) \to 0$ and thus $W_2(\mu_n,\nu) \to W_2(\mu,\nu)$ for all $\nu \in \probt\X$ as $n \to \infty$; for the same reason $H(\mu_n\,|\,\mm) \to H(\mu\,|\,\mm)$. If we also assume that $I(\mu\,|\,\mm) < \infty$, then
\[
\frac{|\nabla\rho_n|^2}{\rho_n} \leq 2\nchi_n^2\frac{|\nabla\rho|^2}{\rho} + 8\rho |\nabla\nchi_n|^2 \leq 2\frac{|\nabla\rho|^2}{\rho} + 8\frac{C^2}{n^2}\rho.
\]
Since the right-hand side is integrable and $\nchi_n \to 1$ as $n \to \infty$, by dominated convergence we get $I(\mu_n\,|\,\mm) \to I(\mu\,|\,\mm)$.

Combining the two steps and using a diagonal argument, the conclusion follows.
\end{proof}

The following conservation result was pointed out in \cite{Conforti17} and \cite{Tamanini17} in the case $\X$ is compact with different approaches; see also \cite{GLR2018}. Following \cite{Tamanini17}, we extend the statement to the present framework.

\begin{lemma}\label{res-h07}
Under Assumptions \ref{ass-01}, for any $\eps > 0$ the function
\[
(0,1) \ni t \mapsto	\int_\X |\nabla\vartheta_t^\eps|^2 \,\d\mu_t^\eps - \frac{\eps^2}{4} I(\mu_t^\eps\,|\,\mm) =: Q^\eps_t
\]
is real-valued and constant. Thus we shall denote it by $Q^\eps$.
\end{lemma}

\begin{proof}
As a first step, for all $0<t<1$ by algebraic manipulation we have
\[
|\nabla\vartheta_t^\eps|^2\rho_t^\eps - \frac{\eps^2}{4}|\nabla\log\rho_t^\eps|^2\rho_t^\eps = -\langle\nabla\varphi_t^\eps,\nabla\psi_t^\eps\rangle \rho_t^\eps = -\eps^2 \langle\nabla f_t^\eps,\nabla g_t^\eps\rangle
\]
so that
\[
Q^\eps_t = -\eps^2\int_\X \langle\nabla f_t^\eps,\nabla g_t^\eps\rangle\,\d\mm.
\]
By Proposition \ref{res-h02} (b-iii) we know that $f^\eps,g^\eps$ are smooth with compact support, hence $|\nabla f^\eps|, |\nabla g^\eps| \in L^\infty(\mm)$ and by the regularization properties of $\hT^\eps$ the same holds for $|\nabla f_t^\eps|, |\nabla g_t^\eps|$: this implies that $Q_t^\eps \in \R$ for all $0<t<1$. As concerns the constancy of $Q_t^\eps$, it is sufficient to prove that the right-hand side above is constant in $t$. To this aim, $\Gamma(f_t^\eps,g_t^\eps)$ is smooth both in time and space, as so are $f_t^\eps,g_t^\eps$; moreover $|\nabla f_t^\eps|, |\nabla g_t^\eps| \in L^\infty(\mm)$ by what just said. Therefore, by dominated convergence and \eqref{eq:heat} we obtain
\[
\begin{split}
\ddt \int_\X \langle\nabla f_t^\eps,\nabla g_t^\eps\rangle\,\d\mm 
& = \int_\X \Big(\langle\nabla\partial_t f_t^\eps,\nabla g_t^\eps\rangle + \langle\nabla f_t^\eps,\nabla\partial_t g_t^\eps\rangle\Big)\,\d\mm \\
& = \eps\int_\X \Big(\langle\nabla\hL f_t^\eps,\nabla g_t^\eps\rangle - \langle\nabla f_t^\eps,\nabla\hL g_t^\eps\rangle\Big)\,\d\mm.
\end{split}
\]
From integration by parts formula it is straightforward to see that the right-hand side vanishes, whence the conclusion.
\end{proof}

Motivated by \eqref{eq-h05}, let us investigate separately the convergence of current and osmotic velocities as $\eps \downarrow 0$.

\begin{lemma}\label{res-h06}
Under Assumptions \ref{ass-01}, for any $\eps > 0$ we have
\begin{equation}\label{eq:toW2}
\lim_{\eps \downarrow 0} \int_{[0,1] \times \X}|\nabla\vartheta_t^\eps|^2\,\dt\d\mu_t^\eps = W_2^2(\mu_0,\mu_1), \quad \lim_{\eps \downarrow 0} \int_{[0,1] \times \X} t|\nabla\vartheta_t^\eps|^2\,\dt\d\mu_t^\eps = \frac{1}{2} W^2_2(\mu_0,\mu_1)
\end{equation}
and 
\begin{equation}\label{eq:to0}
\lim_{\eps \downarrow 0} \eps^2 \int_0^1 I(\mu_t^\eps\,|\,\mm)\,\d t = 0, \qquad\qquad \lim_{\eps \downarrow 0} \eps^2 \int_0^1 t I(\mu_t^\eps\,|\,\mm)\,\d t = 0.
\end{equation}
\end{lemma}

\begin{proof}
Let us first notice that combining \eqref{eq-h05} and \eqref{eq:converge} we get
\[
\lim_{\eps \downarrow 0}\Big( \int_{[0,1] \times \X}|\nabla\vartheta_t^\eps|^2\,\d t\d\mu_t^\eps + \frac{\eps^2}{4} \int_0^1 I(\mu_t^\eps\,|\,\mm)\,\d t\Big) = W_2^2(\mu_0,\mu_1).
\]
Since the continuity equation \eqref{eq:continuity} is satisfied by $\mu_t^\eps$, the Benamou-Brenier formula holds for any $\eps>0$, that is
\[
\int_{[0,1] \times \X}|\nabla\vartheta_t^\eps|^2\,\d t\d\mu_t^\eps \geq W_2^2(\mu_0,\mu_1),
\]
and together with the above limit, this leads us to the first identities both in \eqref{eq:toW2} and \eqref{eq:to0}. From them we immediately deduce that
\begin{equation}\label{eq:Qconv}
\lim_{\eps \downarrow 0} Q^\eps = W_2^2(\mu_0,\mu_1)
\end{equation}
and from the first one in \eqref{eq:to0}
\[
0 \leq \lim_{\eps \downarrow 0} \eps^2 \int_0^1 t I(\mu_t^\eps\,|\,\mm)\,\d t \leq \lim_{\eps \downarrow 0} \eps^2 \int_0^1 I(\mu_t^\eps\,|\,\mm)\,\d t = 0,
\]
whence also the second identity in \eqref{eq:to0}. Finally observe that
\[
\begin{split}
\lim_{\eps \downarrow 0} \int_{[0,1] \times \X} t|\nabla\vartheta_t^\eps|^2\,\dt\d\mu_t^\eps & = \lim_{\eps \downarrow 0} \Big(\int_{[0,1] \times \X} t|\nabla\vartheta_t^\eps|^2\,\dt\d\mu_t^\eps - \int_0^1 t I(\mu_t^\eps\,|\,\mm)\,\d t \Big) \\ & = \lim_{\eps \downarrow 0} \int_0^1 t Q^\eps \dt = W_2^2(\mu_0,\mu_1)\int_0^1 t\,\dt = \frac{1}{2}W_2^2(\mu_0,\mu_1).
\end{split}
\]
\end{proof}

As already noticed at Remark \ref{rem-02}, although it is smooth on the open interval $(0,1)$,  the density  $ \rho ^{ \eps}_t$ might be arbitrarily close to zero. Consequently, the Schrödinger potentials might not be integrable enough, and the Fisher information $I( \mu ^{ \eps}_t\,|\, \mm)$ might behave badly around $t=0$ and $t=1$.  Next lemma provides related regularity and growth controls. Its proof  is strongly inspired by \cite{GigTam17}, \cite{GigTam18}, \cite{GigTam18b} and \cite{Leonard13}; we thus address the reader to these articles for more details.

\begin{lemma}\label{lem:derivatives}
Under Assumptions \ref{ass-01}, let $\delta > 0$ and set, for all $0 \leq t \leq 1$, $\rho_t^{\eps,\delta} := (f_t^\eps + \delta)(g_t^\eps + \delta)$ as well as
\[
\varphi_t^{\eps,\delta} := \eps\log(f_t^\eps + \delta) \qquad \psi_t^{\eps,\delta} := \eps\log(g_t^\eps + \delta) \qquad \vartheta_t^{\eps,\delta} := \frac{1}{2}\big(\psi_t^{\eps,\delta} - \varphi_t^{\eps,\delta}\big).
\]
Then all the functions so defined belong to $C^\infty([0,1] \times \X) \cap W^{1,2}(\X,\mm)$ and, as curves parametrized by $t$ with values in $W^{1,2}(\X,\mm)$, to $AC([0,1],W^{1,2}(\X,\mm))$. The time derivatives of $\varphi_t^{\eps,\delta},\psi_t^{\eps,\delta},\rho_t^{\eps,\delta}$ are given by
\begin{equation}\label{eq:pdedelta}
\begin{split}
\partial_t\varphi_t^{\eps,\delta} = \frac{1}{2}|\nabla\varphi_t^{\eps,\delta}|^2 + \eps\hL \varphi_t^{\eps,\delta} \qquad & \qquad \partial_t\psi_t^{\eps,\delta} = \frac{1}{2}|\nabla\psi_t^{\eps,\delta}|^2 + \eps\hL \psi_t^{\eps,\delta} \\
\partial\rho_t^{\eps,\delta} + {\rm div}_\mm & (\nabla\vartheta_t^{\eps,\delta}\,\rho_t^{\eps,\delta}) = 0
\end{split}
\end{equation}
where $\partial_t\varphi_t^{\eps,\delta}$, $\partial_t\psi_t^{\eps,\delta}$, $\partial_t\rho_t^{\eps,\delta}$ have to be understood both in the classical sense and as strong $W^{1,2}$-limits.

Furthermore, defining $u(z) := z\log z$, the function $t \mapsto \int_\X u(\rho_t^{\eps,\delta})\,\d\mm$ belongs to $C^2([0,1])$ and for every $t \in [0,1]$ it holds
\begin{subequations}
\begin{align}
\label{eq:firstder}
\frac{\d}{\d t}\int_\X u(\rho_t^{\eps,\delta})\,\d\mm & = \int_\X \langle\nabla\rho^{\eps,\delta}_t,\nabla\vartheta^{\eps,\delta}_t\rangle\,\d\mm \\
\label{eq:secondder}
\frac{\d^2}{\d t^2}\int_\X u(\rho_t^{\eps,\delta})\,\d\mm & = \int_\X \Big(\Gamma_2(\vartheta_t^{\eps,\delta}) + \frac{\eps^2}{4}\Gamma_2(\log\rho_t^{\eps,\delta})\Big)\rho_t^{\eps,\delta}\,\d\mm.
\end{align}
\end{subequations}
\end{lemma}

\begin{proof}
As already explained in Section \ref{sec-entint}, under Assumptions \ref{ass-01} $f_t^\eps, g_t^\eps \in C^\infty([0,1],\X) \cap W^{1,2}(\X,\mm)$ and $(f_t^\eps), (g_t^\eps) \in AC([0,1],W^{1,2}(\X,\mm))$. Therefore the regularity and integrability properties of $\varphi_t^{\eps,\delta},\psi_t^{\eps,\delta},\vartheta_t^{\eps,\delta},\rho_t^{\eps,\delta}$ are a straightforward consequence of the chain rule and of the fact that the logarithm is smooth with bounded derivatives on $[\delta,\infty)$. Also the PDEs solved by $\varphi_t^{\eps,\delta},\psi_t^{\eps,\delta},\rho_t^{\eps,\delta}$ are easily deduced, when interpreted in the classical sense, as they follow from \eqref{eq:hjb} and \eqref{eq:continuity}. In order to deduce the same identities with $\partial_t\varphi_t^{\eps,\delta}$, $\partial_t\psi_t^{\eps,\delta}$, $\partial_t\rho_t^{\eps,\delta}$ seen as strong $W^{1,2}$-limits, notice that by the maximum principle
\[
\eps\log\delta \leq \varphi_t^{\eps,\delta} \leq \eps\log(\|f^\eps\|_{L^\infty(\mm)} + \delta), \qquad \forall t \geq 0
\]
whence $(\varphi_t^{\eps,\delta}) \in L^\infty((0,\infty),L^\infty(\mm))$. Moreover, the smoothness of the logarithm, the chain and Leibniz rules entail that
\[
|\nabla\varphi_t^{\eps,\delta}| \leq \eps\frac{|\nabla f_t^\eps|}{\delta} \qquad |\Delta\varphi_t^{\eps,\delta}| \leq \eps\frac{|\Delta f_t^\eps}{\delta} + \eps\frac{|\nabla f_t^\eps|^2}{\delta^2}
\]
whence $(|\nabla\varphi_t^{\eps,\delta}|^2),(\Delta\varphi_t^{\eps,\delta}) \in L^\infty((0,1),W^{1,2}(\X,\mm))$; analogous estimates hold for $\nabla|\nabla\varphi_t^{\eps,\delta}|^2$ and $\nabla\Delta\varphi_t^{\eps,\delta}$. These bounds together with the fact that the PDEs in \eqref{eq:pdedelta} hold in the classical sense imply, by a dominated convergence argument, that \eqref{eq:pdedelta} are satisfied also as strong $W^{1,2}$-limits.

Relying on that, \eqref{eq:firstder} and \eqref{eq:secondder} follow by the computations carried out in \cite{Leonard13}. Indeed, the validity of \eqref{eq:pdedelta} as strong $W^{1,2}$-limits on $[0,1]$ ensures that $t \mapsto u(\rho_t^{\eps,\delta})$ and $t \mapsto \langle\nabla\rho_t^{\eps,\delta},\nabla\vartheta_t^{\eps,\delta}\rangle$ belong to $AC([0,1],L^2(\mm))$ and thus we can pass the time derivatives under the integral sign, i.e.
\begin{equation}\label{eq:underintegral}
\frac{\d}{\d t}\int_\X u(\rho_t^{\eps,\delta})\,\d\mm = \int_\X \partial_t u(\rho_t^{\eps,\delta})\,\d\mm, \qquad \frac{\d}{\d t}\int_\X \langle\nabla\rho_t^{\eps,\delta},\nabla\vartheta_t^{\eps,\delta}\rangle\,\d\mm = \int_\X \partial_t \langle\nabla\rho_t^{\eps,\delta},\nabla\vartheta_t^{\eps,\delta}\rangle\,\d\mm.
\end{equation}
Then the Hamilton-Jacobi-Bellman equations for the Schrödinger potentials and the continuity equation for the entropic interpolation together with $\eps\log\rho_t^\eps = \varphi_t^\eps + \psi_t^\eps$, here replaced by $\eps\log\rho_t^{\eps,\delta} = \varphi_t^{\eps,\delta} + \psi_t^{\eps,\delta}$, are the only tools needed to deduce \eqref{eq:firstder} and \eqref{eq:secondder}. Finally, the fact that:
\begin{enumerate}[-]
\item $(\rho_t^{\eps,\delta}),(\vartheta_t^{\eps,\delta}) \in AC([0,1],W^{1,2}(\X,\mm))$;
\item $(\rho_t^{\eps,\delta}) \in AC([0,1],W^{1,2}(\X,\mm))$, $(|\nabla\vartheta_t^{\eps,\delta}|^2),(\Delta\vartheta_t^{\eps,\delta}), (|\nabla\log\rho_t^{\eps,\delta}|^2),(\Delta\log\rho_t^{\eps,\delta})$ belong to $L^\infty((0,1),W^{1,2}(\X,\mm))$ and $\vartheta_t^{\eps,\delta},\log\rho_t^{\eps,\delta} \in C^\infty([0,1] \times \X)$;
\end{enumerate}
imply the continuity on $[0,1]$ of the right-hand sides of \eqref{eq:firstder} and \eqref{eq:secondder} respectively.
\end{proof}

With this results at disposal we can prove our main lemma: a rigorous `entropic' analogue of Otto-Villani's heuristic argument.

\begin{lemma}\label{res-h09}
Under Assumptions \ref{ass-01}, for any $\eps > 0$ it holds
\begin{equation}\label{eq-h11}
\begin{split}
H(\mu_1\,|\,\mm) - H(\mu_0\,|\,\mm)	\leq \int_\X \langle\nabla\vartheta_1^\eps,\nabla\rho_1\rangle\,\d\mm & - \kappa \int_{[0,1] \times \X} t|\nabla\vartheta_t^\eps|^2\, \d\mu_t^\eps\dt \\
& - \kappa\frac{\eps^2}{4} \int_{[0,1] \times \X} t|\nabla\log\rho_t^\eps|^2\, \d\mu_t^\eps\dt.
\end{split}
\end{equation}
\end{lemma}

\begin{proof}
The proof of the lemma is based on the standard calculus identity
\begin{equation}\label{eq-h12}
h(1) = h(0) + h'(1) -\int_0^1 t h''(t)\,\dt,
\end{equation}
valid for any $C^2$-regular function $h$, applied to $t \mapsto \int_\X u(\rho_t^{\eps,\delta})\,\d\mm$ defined as in Lemma \ref{lem:derivatives}. Plugging \eqref{eq:firstder} and \eqref{eq:secondder} into \eqref{eq-h12} and using the well-known inequalities
\[
\Gamma_2(\vartheta_t^{\eps,\delta}) \geq \kappa |\nabla\vartheta_t^{\eps,\delta}|^2 \qquad\qquad \Gamma_2(\log\rho_t^{\eps,\delta}) \geq \kappa |\nabla\log\rho_t^{\eps,\delta}|^2,
\]
consequence of the lower Ricci bounds \eqref{eq-cd(k,inf)} and \eqref{eq-cd(k,n)} rewritten in the form of the Bochner-Lichnerowicz-Weitzenböck formula, we obtain
\begin{equation}\label{eq:ineqwithdelta}
\begin{split}
\int_\X u(\rho_1^{\eps,\delta})\,\d\mm - \int_\X u(\rho_0^{\eps,\delta})\,\d\mm	\leq \int_\X \langle\nabla\vartheta_1^{\eps,\delta},\nabla\rho_1^{\eps,\delta}\rangle\,\d\mm & - \kappa \int_{[0,1] \times \X} t|\nabla\vartheta_t^{\eps,\delta}|^2\, \d\mu_t^{\eps,\delta}\dt \\
& - \kappa\frac{\eps^2}{4} \int_{[0,1] \times \X} t|\nabla\log\rho_t^{\eps,\delta}|^2\, \d\mu_t^{\eps,\delta}\dt,
\end{split}
\end{equation}
where $\mu_t^{\eps,\delta} := \rho_t^{\eps,\delta}\mm$. It is now sufficient to pass to the limit as $\delta \downarrow 0$. By dominated convergence (recall that $f^\eps,g^\eps \in L^\infty(\mm)$) it is easy to see that the left-hand side converges to $H(\mu_1\,|\,\mm) - H(\mu_0\,|\,\mm)$. In order to apply the dominated convergence theorem also to the first term on the right-hand side, notice that
\[
\nabla\rho_1^{\eps,\delta} = \nabla\rho_1 + \delta\nabla f^\eps_1 + \delta\nabla g^\eps \qquad\textrm{and}\qquad \nabla\vartheta_1^{\eps,\delta} = \frac{1}{2}\nabla\psi_1^{\eps,\delta} - \frac{1}{2}\nabla\varphi_1^{\eps,\delta}.
\]
Thus, to control $\langle\nabla\vartheta_1^{\eps,\delta},\nabla\rho_1^{\eps,\delta}\rangle$ let us first observe that
\[
|\langle\nabla\psi_1^{\eps,\delta},\nabla\rho_1\rangle | = \eps\frac{|\nabla g^\eps||\nabla\rho_1|}{g^\eps + \delta} \leq \eps\frac{|\nabla g^\eps||\nabla\rho_1|}{g^\eps} \leq \eps f_1^\eps\frac{|\nabla g^\eps|^2}{g^\eps} + \eps|\nabla g^\eps||\nabla f_1^\eps|
\]
and remark that the right-hand side is integrable by Proposition \ref{res-h02}-(c); secondly
\[
|\langle\nabla\varphi_1^{\eps,\delta},\nabla\rho_1\rangle | \leq \eps\frac{|\nabla f_1^\eps||\nabla\rho_1|}{f_1^\eps} \leq \eps g^\eps\frac{|\nabla f_1^\eps|^2}{f_1^\eps} + \eps|\nabla g^\eps||\nabla f_1^\eps|
\]
and in this case the right-hand side is integrable as $g^\eps$ has compact support and $f_1^\eps$ is bounded away from 0 therein; finally,
\[
\delta |\langle\nabla\psi_1^{\eps,\delta},\nabla f_1^\eps\rangle | = \eps\delta\frac{|\nabla g^\eps||\nabla f_1^\eps|}{g^\eps + \delta} \leq \eps |\nabla g^\eps||\nabla f_1^\eps|
\]
as $g^\eps + \delta \geq \delta$ and the same strategy applies to all the remaining terms that we obtain developing $\langle\nabla\vartheta_1^{\eps,\delta},\nabla\rho_1^{\eps,\delta}\rangle$. Therefore, the first term on the right-hand side of \eqref{eq:ineqwithdelta} converges to the first term on the right-hand side of \eqref{eq-h11}. As regards the other two summands, by the very definition of $\vartheta_t^{\eps,\delta}$ and since
\[
\eps\log\rho_t^{\eps,\delta} = \varphi_t^{\eps,\delta} + \psi_t^{\eps,\delta},
\]
the conclusion will follow if we are able to prove that
\begin{subequations}
\begin{align}
\label{eq:varphi}
& \lim_{\delta \downarrow 0} \int_{[0,1] \times \X} |\nabla\varphi_t^{\eps,\delta}|^2\, \d\mu_t^{\eps,\delta}\dt = \int_{[0,1] \times \X} |\nabla\varphi_t^\eps|^2\, \d\mu_t^\eps\dt, \\
\label{eq:psi}
& \lim_{\delta \downarrow 0} \int_{[0,1] \times \X} |\nabla\psi_t^{\eps,\delta}|^2\, \d\mu_t^{\eps,\delta}\dt = \int_{[0,1] \times \X} |\nabla\psi_t^\eps|^2\, \d\mu_t^\eps\dt.
\end{align}
\end{subequations}
To this aim, notice that $\rho_t^{\eps,\delta} = \rho_t^\eps + \delta f_t^\eps + \delta g_t^\eps + \delta^2$, whence using either $f_t^\eps + \delta \geq f_t^\eps$ or $f_t^\eps + \delta \geq \delta$ it is easy to infer that
\[
\begin{split}
|\nabla\varphi_t^{\eps,\delta}|^2\rho_t^\eps & = \eps^2\frac{|\nabla f_t^\eps|^2}{(f_t^\eps + \delta)^2}\rho_t^\eps \leq \eps^2\frac{|\nabla f_t^\eps|^2}{(f_t^\eps)^2}\rho_t^\eps = |\nabla\varphi_t^\eps|^2\rho_t^\eps, \\
\delta|\nabla\varphi_t^{\eps,\delta}|^2 f_t^\eps & = \delta\eps^2\frac{|\nabla f_t^\eps|^2}{(f_t^\eps + \delta)^2} f_t^\eps \leq \eps^2 |\nabla f_t^\eps|^2, \\
\delta|\nabla\varphi_t^{\eps,\delta}|^2 g_t^\eps & = \delta\eps^2\frac{|\nabla f_t^\eps|^2}{(f_t^\eps + \delta)^2} g_t^\eps \leq \eps^2\frac{|\nabla f_t^\eps|^2}{f_t^\eps} g_t^\eps = |\nabla\varphi_t^\eps|^2\rho_t^\eps, \\
\delta^2|\nabla\varphi_t^\eps|^2 & = \delta^2\eps^2\frac{|\nabla f_t^\eps|^2}{(f_t^\eps + \delta)^2} \leq \eps^2|\nabla f_t^\eps|^2.
\end{split}
\]
All the right-hand sides above are integrable on $[0,1] \times \X$ (either by \eqref{eq:integrability} or by the Bakry-\'Emery contraction estimate \eqref{eq:bakry} together with $f^\eps \in C^\infty_c(\X)$), thus by dominated convergence \eqref{eq:varphi} follows. An analogous argument holds for \eqref{eq:psi}, whence the conclusion.
\end{proof}

\subsection*{Completion of the proof of Proposition \ref{res-h04}}

We are now ready to complete the proof of the HWI* inequality.

\begin{proof}[Proof of Proposition \ref{res-h04}]\label{page-proof}
First of all, by Proposition \ref{res-h02} and Lemma \ref{res-h01}, it is sufficient to prove the result for any $\mu_0,\mu_1$ satisfying Assumptions \ref{ass-01}. This allows us to invoke Lemma \ref{res-h09}. Secondly, observe that by the Cauchy-Schwarz inequality
\[
\int_\X \langle\nabla\vartheta_1^\eps,\nabla\rho_1\rangle\,\d\mm = \int_\X \langle\nabla\vartheta_1^\eps,\nabla\log\rho_1\rangle\d\mu_1 \leq \int_\X |\nabla\vartheta_1^\eps|^2\d\mu_1 \sqrt{I(\mu_1\,|\,\mm)},
\]
so that if we plug this information into \eqref{eq-h11} we obtain
\[
\begin{split}
H(\mu_1\,|\,\mm) - H(\mu_0\,|\,\mm)	\leq \int_\X |\nabla\vartheta_1^\eps|^2\d\mu_1 \sqrt{I(\mu_1\,|\,\mm)} & - \kappa \int_{[0,1] \times \X} t|\nabla\vartheta_t^\eps|^2\, \d\mu_t^\eps\dt \\
& - \kappa\frac{\eps^2}{4} \int_{[0,1] \times \X} t|\nabla\log\rho_t^\eps|^2\, \d\mu_t^\eps\dt.
\end{split}
\]
Let us now pass to the limit as $\eps \downarrow 0$: by Lemma \ref{res-h07}, at $t=1$ we have
\[
Q^\eps = \int_\X |\nabla\vartheta_1^\eps|^2 \,\d\mu_1 - \frac{\eps^2}{4} I(\mu_1\,|\,\mm),
\]
so that \eqref{eq:Qconv} together with $I(\mu_1\,|\,\mm) < \infty$ yields
\[
\lim_{\eps \downarrow 0}\int_\X |\nabla\vartheta_1^\eps|^2 \d\mu_1 = W_2^2(\mu_0,\mu_1).
\]
By Lemma \ref{res-h06} we can pass to the limit as $\eps \downarrow 0$ also in the remaining terms on the right-hand side, thus concluding.
\end{proof}

Let us mention that another approach to prove the HWI inequality via the Schr\"odinger problem is pointed out in a recent work \cite{CR2018} where an entropic counterpart of the HWI inequality is formally obtained by differentiating the convexity estimate of the entropy along the entropic interpolations introduced in \cite[Thm.\,1.4]{Conforti17}.

\section{Final remarks and comments}
\label{sec-remarks}

\subsection*{Lagrangian and Hamiltonian interpretation}

From a heuristic point of view, the expression of the constant quantity $Q^\eps$ can be deduced by standard arguments in Lagrangian and Hamiltonian formalism. Indeed, motivated by \eqref{eq-h05} let us consider the action functional
\begin{equation}\label{eq:action}
\mathscr{A}(\nu,v) = \int_{[0,1] \times \X} \Big(\frac{|v_t|^2}{2} + \frac{\eps^2}{8}|\nabla\log\nu_t|^2\Big)\nu_t \,\dt\d\mm
\end{equation}
associated to the Lagrangian
\[
\mathscr{L}(\nu,v) = \int_\X \Big(\frac{|v|^2}{2}\nu + \frac{\eps^2}{8}\frac{|\nabla\nu|^2}{\nu}\Big) \,\d\mm.
\]
By means of Legendre's transform, the corresponding Hamiltonian is given by
\[
\mathscr{H}(\nu,p) = \int_\X \Big(\frac{|p|^2}{2\nu} - \frac{\eps^2}{8}\frac{|\nabla\nu|^2}{\nu}\Big) \,\d\mm
\]
and, at least formally, $\mathscr{H}$ is constant along the critical points of $\mathscr{A}$. Since $\mu_0$ and $\mu_1$ are prescribed, the Euler-Lagrange equation for \eqref{eq:action} reads as
\[
\begin{split}
& \partial_t\nu_t + {\rm div}_\mm(\nu_t v_t) = 0 \\
& \partial_t v_t + \frac{|v_t|^2}{2} = -\frac{\eps^2}{8}\big(2\Delta\log\nu_t + |\nabla\log\nu_t|^2\big)
\end{split}
\]
and, as it is not difficult to see (e.g.\ following the computations carried out in \cite{GigTam18} which fit to  Setting  \ref{setting}), these PDEs are satisfied along $(\rho_t^\eps,\vartheta_t^\eps)$. Finally, as in the Hamiltonian $p$ represents a momentum density, it is natural to set $p_t := \nu_t v_t$. From these considerations, the guess on the existence of the conserved quantity of Lemma \ref{res-h07} and on its expression follows. \\
This point of view is particularly investigated in the recent paper \cite{GLR2018}, to which we refer for more details.

\subsection*{The compact case}

As already mentioned in Remark \ref{rem-02}, in general Schrödinger potentials fail to be smooth in $t=0$ or $t=1$ and are not even integrable for any $0 < t < 1$, because $f_t^\eps,g_t^\eps$ can be arbitrarily close to 0. This problem can be overcome if $f^\eps,g^\eps \geq c > 0$ for some constant $c$: then $f_t^\eps,g_t^\eps \geq c$ as well by the maximum principle and $\varphi_t^\eps,\psi_t^\eps \geq \eps\log c$ from their very definition. However, even if we assume $\mu_0,\mu_1$ to have densities bounded away from 0 (thus forcing $\mm$ to belong to $\probt\X$, a condition which does not follow from our Setting  \ref{setting}, unless $\kappa > 0$, cf.\ Remark \ref{rem-01}-(a)), it is not known yet whether this implies an analogous lower bound on $f^\eps$ and $g^\eps$. This is the motivation behind the definitions provided in Lemma \ref{lem:derivatives}.

On the contrary, if $\X$ is assumed to be compact, then it is not restrictive to assume $\mu_0,\mu_1 \geq c\mm$ for some $c>0$: indeed any $\mu \in \prob\X$ can be approximated by
\[
\mu_n := \alpha_n\rho_n\mm \qquad \textrm{with } \rho_n := \max\{\rho,1/n\}
\]
(where $\alpha_n$ is the renormalization constant) in $W_2$, entropy and Fisher information in the sense of Lemma \ref{res-h01}. Secondly, $\hr_\eps$ is bounded from above since $\hr_\eps \in C^\infty(\X)$. These two facts imply that $f^\eps,g^\eps \geq c' > 0$, as shown in \cite{GigTam17}. As a consequence the proof of the HWI inequality is easier and the parallelism with the heuristic proof of Otto-Villani \cite{OttoVillani00} is stronger, since no `$\delta$-argument' in Lemma \ref{lem:derivatives} and Lemma \ref{res-h09} is needed; in particular, Lemma \ref{lem:derivatives} is already known to hold by \cite{GigTam17} and thus no proof is required.

\subsection*{\texorpdfstring{$\RCD$}{RCD} spaces}

All the results presented in this paper, with particular mention of the key Lemma \ref{res-h09} and Proposition \ref{res-h04}, are also true in a more general framework than the one of Setting  \ref{setting}-(b), namely in $\RCD^*(K,N)$ spaces (introduced in \cite{Gigli12}). If $\X$ is assumed to be compact, then this has been shown by the last author in his PhD thesis \cite{Tamanini17}. When $\X$ is not compact, the most important steps in our entropic approach still hold. Namely:
\begin{enumerate}[-]
\item all the regularity and integrability results concerning Schrödinger potentials and entropic interpolations mentioned in Section \ref{sec-entint} as well as the dynamic representation of the entropic cost;
\item the regularizing and contraction properties of $(\hT_t)$;
\item the existence of `good' cut-off functions;
\item the Benamou-Brenier formula and the Bochner-Lichnerowicz-Weitzenböck inequality.
\end{enumerate}
 The reader is addressed to \cite{GigTam18}, \cite{GigTam18b} for the first point and to \cite{Gigli14} for all the others.

\bigskip

\noindent{\bf Acknowledgements.}  This research was supported  by the French ANR-17-CE40-0030 EFI project and the LABEX MILYON
ANR-10-LABX-0070.

%\bibliographystyle{siam}
%\bibliography{HWI}

\begin{thebibliography}{10}

\bibitem{AmbrosioGigliSavare08}
{\sc L.~Ambrosio, N.~Gigli, and G.~Savar{\'e}}, {\em Gradient flows in metric
  spaces and in the space of probability measures}, Lectures in Mathematics ETH
  Z\"urich, Birkh\"auser Verlag, Basel, second~ed., 2008.

\bibitem{AmbrosioGigliSavare11}
\leavevmode\vrule height 2pt depth -1.6pt width 23pt, {\em Calculus and heat
  flow in metric measure spaces and applications to spaces with {R}icci bounds
  from below}, Invent. Math., 195 (2014), pp.~289--391.

\bibitem{AFLMR07}
{\sc D.~Azagra, J.~Ferrera, F.~L{\'o}pez-Mesas, and Y.~Rangel}, {\em Smooth
  approximation of {L}ipschitz functions on {R}iemannian manifolds}, J. Math.
  Anal. Appl., 326 (2007), pp.~1370--1378.

\bibitem{BakryGentilLedoux14}
{\sc D.~Bakry, I.~Gentil, and M.~Ledoux}, {\em Analysis and geometry of
  {M}arkov diffusion operators}, vol.~348 of Grundlehren der Mathematischen
  Wissenschaften [Fundamental Principles of Mathematical Sciences], Springer,
  Cham, 2014.

\bibitem{CGP14}
{\sc Y.~Chen, T.~Georgiou, and M.~Pavon}, {\em On the relation between optimal
  transport and {S}chr{\"o}dinger bridges: {A} stochastic control viewpoint},
  Journal of Optimization Theory and Applications, 169 (2016), pp.~671--691.

\bibitem{Conforti17}
{\sc G.~{Conforti}}, {\em A second order equation for {S}chr{\"o}dinger bridges
  with applications to the hot gas experiment and entropic transportation
  cost}.
\newblock Preprint, arXiv:1704.04821, (2017).

\bibitem{CR2018}
{\sc G.~{Conforti} and L.~{Ripani}}, {\em Around the entropic {T}alagrand
  inequality}.
\newblock Work in progress.

\bibitem{GLR2018}
{\sc I.~Gentil, C.~L{\'e}onard, and L.~Ripani}, {\em Dynamical aspects of
  generalized {S}chr{\"o}dinger problem via {O}tto calculus - {A} heuristic
  point of view}.
\newblock Preprint, arXiv:1806.01553 (2018).

\bibitem{GLR15}
\leavevmode\vrule height 2pt depth -1.6pt width 23pt, {\em About the analogy
  between optimal transport and minimal entropy}, Ann. Fac. Toulouse, 26
  (2017), pp.~569--600.

\bibitem{Gigli14}
{\sc N.~Gigli}, {\em Nonsmooth differential geometry - an approach tailored for
  spaces with {R}icci curvature bounded from below}.
\newblock Accepted at Mem. Amer. Math. Soc., arXiv:1407.0809, 2014.

\bibitem{Gigli12}
\leavevmode\vrule height 2pt depth -1.6pt width 23pt, {\em On the differential
  structure of metric measure spaces and applications}, Mem. Amer. Math. Soc.,
  236 (2015), pp.~vi+91.

\bibitem{GigTam17}
{\sc N.~Gigli and L.~Tamanini}, {\em Second order differentiation formula on
  compact ${RCD}^*({K},{N})$ spaces}.
\newblock Preprint, arXiv:1701.03932, (2017).

\bibitem{GigTam18}
{\sc N.~{Gigli} and L.~{Tamanini}}, {\em Second order differentiation formula
  on ${RCD}^*({K},{N})$ spaces}.
\newblock Preprint, arXiv:1802.02463, (2018).

\bibitem{GigTam18b}
{\sc N.~Gigli and L.~Tamanini}, {\em Benamou-{B}renier and duality formulas for
  the entropic cost on ${RCD}^*({K},{N})$ spaces}, Preprint, arXiv:1805.06325,
  (2018).

\bibitem{Grigoryan09}
{\sc A.~Grigoryan}, {\em Heat kernel and analysis on manifolds}, vol.~47,
  American Mathematical Soc., 2009.

\bibitem{Leonard12}
{\sc C.~L{\'e}onard}, {\em From the {S}chr\"odinger problem to the
  {M}onge-{K}antorovich problem}, J. Funct. Anal., 262 (2012), pp.~1879--1920.

\bibitem{Leonard14b}
\leavevmode\vrule height 2pt depth -1.6pt width 23pt, {\em Some properties of
  path measures}, S{\'e}minaire de Probabilit{\'e}s XLVI,  (2014),
  pp.~207--230.

\bibitem{Leonard14}
\leavevmode\vrule height 2pt depth -1.6pt width 23pt, {\em A survey of the
  {S}chr\"odinger problem and some of its connections with optimal transport},
  Discrete Contin. Dyn. Syst., 34 (2014), pp.~1533--1574.

\bibitem{Leonard13}
{\sc C.~L{\'e}onard}, {\em On the convexity of the entropy along entropic
  interpolations}, in Measure Theory in Non-Smooth Spaces, N.~Gigli, ed.,
  Partial Differential Equations and Measure Theory, De Gruyter Open, 2017.
\newblock % Preprint, arXiv:1310.1274.

\bibitem{Mikami04}
{\sc T.~Mikami}, {\em Monge's problem with a quadratic cost by the zero-noise
  limit of {$h$}-path processes}, Probab. Theory Related Fields, 129 (2004),
  pp.~245--260.

\bibitem{Nelson67}
{\sc E.~Nelson}, {\em Dynamical theories of {B}rownian motion}, Princeton
  University Press, Princeton, N.J., 1967.

\bibitem{Otto01}
{\sc F.~Otto}, {\em The geometry of dissipative evolution equations: the porous
  medium equation}, Comm. Partial Differential Equations, 26 (2001),
  pp.~101--174.

\bibitem{OttoVillani00}
{\sc F.~Otto and C.~Villani}, {\em Generalization of an inequality by
  {T}alagrand and links with the logarithmic {S}obolev inequality}, J. Funct.
  Anal., 173 (2000), pp.~361--400.

\bibitem{Tamanini17}
{\sc L.~Tamanini}, {\em Analysis and geometry of {RCD} spaces via the
  {S}chr{\"o}dinger problem}, PhD thesis, {U}niversit{\'e} {P}aris {N}anterre
  and SISSA, 2017.

\bibitem{Villani09}
{\sc C.~Villani}, {\em Optimal transport. Old and new}, vol.~338 of Grundlehren
  der Mathematischen Wissenschaften, Springer-Verlag, Berlin, 2009.

\end{thebibliography}

\end{document}